\documentclass[a4paper,12pt]{article}

\usepackage{amsmath}
\usepackage{amsthm}
\usepackage{epsfig}
\usepackage{amssymb}
\usepackage{latexsym}
\usepackage[T1]{fontenc}
\usepackage[utf8]{inputenc}
\usepackage{color}

\title{Short cycle covers on cubic graphs using chosen $2$-factor}
\author{
Barbora Candráková,
Robert Lukoťka
\\[3mm]
\\{\tt candrakova@dcs.fmph.uniba.sk}
\\{\tt lukotka@dcs.fmph.uniba.sk}
\\[5mm]
Faculty of Mathematics, Physics and Informatics\\
Comenius University \\
Mlynská dolina, 842 48 Bratislava
}

\theoremstyle{definition}
\newtheorem{definition}{Definition}
\numberwithin{definition}{section}
\newtheorem{theorem}[definition]{Theorem}
\newtheorem{corollary}[definition]{Corollary}
\newtheorem{lemma}[definition]{Lemma}
\newtheorem{conjecture}[definition]{Conjecture}

\let\epsilon=\varepsilon

\let\phi = \varphi

\begin{document}

\maketitle

\abstract{ 
We show that every bridgeless cubic graph $G$ with $m$ edges has a cycle cover of length at most $1.6 m$. Moreover, if $G$ does not contain any intersecting circuits of length $5$, then $G$ has a cycle cover of length $212/135 \cdot m \approx 1.570 m$ and if $G$ contains no $5$-circuits, then it has a cycle cover of length at most $14/9 \cdot m \approx 1.556 m$.
To prove our results, we show that each $2$-edge-connected cubic graph $G$ on $n$ vertices has a $2$-factor containing at most $n/10+f(G)$ circuits of length $5$, where the value of $f(G)$ only depends on the presence of several subgraphs arising from the Petersen graph. 
As a corollary we get that each $3$-edge-connected cubic graph on $n$ vertices has 
a $2$-factor containing at most $n/9$ circuits of length $5$ and
each $4$-edge-connected cubic graph on $n$ vertices has 
a $2$-factor containing at most $n/10$ circuits of length $5$.
}

\section{Introduction}

A \emph{cycle cover} of a graph $G=(V,E)$ is a collection of \emph{cycles} in $G$ such that each edge of the graph is contained in at least one of the cycles. (By the term \emph{cycle} we understand a graph whose  all vertices are of even degree.) If an edge of $G$ is a bridge, then it does not belong to any cycle. Thus only bridgeless graphs are of interest regarding cycle covers. Szekeres \cite{Szekeres} and Seymour \cite{Seymour} conjectured that this condition is sufficient for a graph to have a cycle cover such that each edge is covered exactly twice.
\begin{conjecture}[Cycle Double Cover Conjecture]
\label{cdcc}
Every bridgeless graph has a collection of cycles covering each edge exactly twice.
\end{conjecture}

The \emph{length of a cycle cover} is the sum of all cycle lengths in the cover. In this paper, we are interested in finding cycle covers such that their lengths are as small as possible. Alon and Tarsi \cite{AT} conjectured the following bound on the length of the shortest cycle cover.
\begin{conjecture}[Short Cycle Cover Conjecture]
\label{sccc}
Every bridgeless graph with $m$ edges has cycle cover of length at most $1.4m$.
\end{conjecture}
The shortest cycle cover of the Petersen graph has length $21$ and consists of $4$ cycles. 
The upper bound given by Conjecture \ref{sccc} is,  therefore, sharp
and the constant $1.4$ cannot be improved.

The Short Cycle Cover Conjecture relates to several other well known conjectures for example, as shown in \cite{JT}, it implies the Cycle Double Cover Conjecture.
Jamshy, Raspaud, and Tarsi \cite{JRT} showed that graphs with nowhere-zero $5$-flow, which,  according the Tutte's $5$-flow conjecture, are all bridgeless graphs, have a cycle cover of length at most $1.6m$. Máčajová et al. \cite{MRTZ} showed that the existence of a Fano-flow using at most $5$ lines of the Fano plane on bridgeless cubic graphs 
implies the existence of a cycle cover of length at most $1.6m$.
They also showed that the existence of a Fano-flow using at most $4$ lines of the Fano plane on bridgeless cubic graphs, 
which is a consequence of the Fulkerson Conjecture,
implies the existence of a cycle cover of length at most $14/9 \cdot m \approx 1.556m$.

The best known general result on short cycle covers is due to Alon and Tarsi \cite{AT} and Bermond, Jackson, and Jaeger \cite{BJJ} who proved that every bridgeless graph with $m$ edges has a cycle cover of total length at most $5/3 \cdot m$.
On the other hand, there are numerous results on short cycle covers for special classes of graphs. We refer the reader to Chapter 8 of the book \cite{CQbook} by Zhang, which is devoted to short cycle covers.
Significant attention has been devoted to cubic graphs since any result on short cycle covers 
of subcubic or weighted cubic graphs can be extended to a cycle cover of general graphs.
The best result for cubic graphs up to date is by Kaiser et al. \cite{kaiser}. They proved that there exists a cycle cover of length at most $34/21 \cdot m\approx1.619m$ for each bridgeless cubic graph and there exists a cycle cover of length at most $44/27 \cdot m \approx1.630m$ for each bridgeless graph with minimum degree $3$.

The main obstacle not allowing to improve the bound in the approach of Kaiser et al. \cite{kaiser} are the circuits of length $5$ contained in the graph.
Indeed, Hou and Zhang \cite{HZ} proved that if $G$ is a bridgeless cubic graph such that all circuits of length $5$ are disjoint, then $G$ has a cycle cover of length at most $351/225 \cdot m \approx 1.6044m$.
Moreover, if $G$ contains no circuits of length $5$, then the length of the cover is at most $1.6m$.

We slightly refine the approach of Kaiser et al. \cite{kaiser} and combine it with a technique for avoiding certain number of $5$-circuits in $2$-factors of cubic graphs introduced in \cite{LMMS} and \cite{CL}, which allows us to improve the bound on the length of the shortest cycle cover as follows.

\newtheorem*{thmmain}{Theorem~\ref{thmmain}}
\begin{thmmain}
Every bridgeless cubic graph with $m$ edges has a cycle cover of length at most $1.6m$.
\end{thmmain}

Due to our technique we can state the results in terms of the number of circuits of length $5$ that are contained in the graph, improving the results of Hou and Zhang for cubic graphs with restrictions on circuits of length $5$.
\newtheorem*{thm5Circuit}{Theorem~\ref{thm5Circuit}}
\begin{thm5Circuit}
Every bridgeless cubic graph with $m$ edges and at most $k$ circuits of length $5$
has a cycle cover of length at most 
$14/9 \cdot m + 1/9 \cdot k$.
\end{thm5Circuit}

\newtheorem*{corDisjoint5}{Corollary~\ref{corDisjoint5}}
\begin{corDisjoint5}
Every bridgeless cubic graph with $m$ edges, such that all circuits of length $5$ are disjoint, 
has a cycle cover of length at most $212/135 \cdot m \approx 1.570 m$.
\end{corDisjoint5}

\newtheorem*{corWithout5}{Corollary~\ref{corWithout5}}
\begin{corWithout5}
Every bridgeless cubic graph with $m$ edges without circuits of length $5$ has a cycle cover of length at most $14/9 \cdot m \approx 1.556 m$.
\end{corWithout5}

To prove our results we modify the statement of the main theorem from \cite{CL}, which deals with the $5$-circuits in $2$-factors of cubic graphs,  to match the requirements of our proof. 
Indeed, we prove that each $2$-edge-connected cubic graph $G$ on $n$ vertices has a $2$-factor containing at most $n/10+f(G)$ circuits of length $5$, where the value of $f(G)$ depends on the presence of several subgraphs arising from the Petersen graph. This improvement is interesting on its own, as from the reformulated statement we can obtain several corollaries that were not available before.

\newtheorem*{cor3EdgeConn}{Corollary~\ref{cor3EdgeConn}}
\begin{cor3EdgeConn}
Let $G$ be a cyclically $3$-edge-connected cubic graph on $n$ vertices other than the Petersen graph. 
Then $G$ has a $2$-factor with at most $1/9 \cdot n$ circuits of length $5$. 
\end{cor3EdgeConn}
\newtheorem*{cor4EdgeConn}{Corollary~\ref{cor4EdgeConn}}
\begin{cor4EdgeConn}
Let $G$ be a cyclically $4$-edge-connected cubic graph on $n$ vertices other than the Petersen graph. 
Then $G$ has a $2$-factor with at most $1/10 \cdot n$ circuits of length $5$. 
\end{cor4EdgeConn}

\section{Avoiding circuits of length $5$}
In this section we refine the results from \cite{CL} where we proved that every $2$-edge-connected cubic graph on $n$ vertices has a $2$-factor containing at most $2(n-2)/15$ circuits of length $5$. This bound is attained for infinitely many graphs, however, as we show, we can decrease this number to $n/10$ in certain cases.
Although, the statement of our theorem is different than the statement in \cite{CL},
large part of our proof is identical to the proof in \cite{CL}. We will omit several routine steps, refering the reader to \cite{CL}, to focus on the main ideas.

Let $G$ be a $2$-edge-connected cubic graph and let $E(G)=\{e_1, e_2, \ldots, e_m\}$. 
The \emph{boundary} of a set $S \subseteq V(G)$, denoted by $\delta(S)$, is the set of edges with exactly one endpoint in $S$.
We can represent each perfect matching $M$ of $G$ as a characteristic vector ${\bf p}=(p_1,p_2,\ldots,p_m)$ from ${\mathbb R}^m$ such that $p_i=1$ if $e_i \in M$, and $p_i=0$ otherwise. For convenience we 
identify perfect matchings
and their characteristic vectors when no confusion can occur. 
The \emph{perfect matching polytope} ${\cal M}(G)$ of a graph $G$ is the convex hull in ${\mathbb R}^m$ of the set of characteristic vectors of all perfect matchings of $G$. 
For any point ${\bf p}\in{\cal M}(G)$ we denote by $p_e$ the value of the coordinate of ${\bf p}$
corresponding to the edge $e$.  Equivalently, ${\cal M}(G)$ can be described as a set of vectors from ${\mathbb R}^m$ such that the following properties hold \cite{edmonds}: 
\begin{eqnarray*}
p_e &\ge& 0 \hspace{2em} \text{for each } e\in E(G),\\
\sum_{e\in \delta(\{v\})} p_e &=&1 \hspace{2em} \text{for each } v\in V(G), \\
\sum_{e\in \delta(S)} p_e &\ge& 1 \hspace{2em} \text{for each } S\subseteq V(G) \text{ with } |S| \text{ odd}. 
\end{eqnarray*}
The main idea of our proof consists of defining a linear function on ${\cal M}(G)$
and minimizing it. 
It is well-known that the minimum of a linear function over a polytope is attained at some vertex 
of the polytope, which is a characteristic vector of a perfect matching \cite{edmonds}.
Therefore, there exists a perfect matching that attains the optimal value. To bound the optimal value we use the fact that the point $(1/3, 1/3, \ldots, 1/3)$ lies within ${\cal M}(G)$ for each bridgeless cubic graph $G$. 

Let $G$ be a bridgeless cubic graph and let $F$ be a $2$-factor of $G$. Let $I(F,G)$ be the number of $5$-circuits in $F$. 
We say that a graph is \emph{colourable} if it admits a $3$-edge-colouring, otherwise it is \emph{uncolourable}.
Similarly as in \cite{CL}, we consider small uncolourable subgraphs of $G$ which can be separated by a $2$- or $3$-edge-cut. Let $C$ be a $3$-circuit of $G$. We can \emph{reduce} $C$ by contracting it into a vertex $v$.
We call the reverse operation as \emph{inserting a triangle} into $v$.
Let $C$ be a $2$-circuit of $G$. We can \emph{reduce} $C$ by deleting it and by adding a new edge $e$
that connects the two new vertices of degree $2$.
We call the reverse operation as \emph{inserting a 2-circuit} into $e$.

\begin{figure}[htp]
\center
\begin{tabular}{ccccc}
\includegraphics{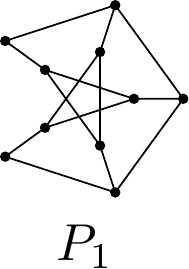} & & & &\includegraphics{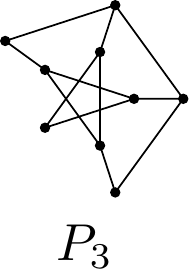}\\
\end{tabular}
\caption{Graphs $P_2$ and $P_3$.}
\label{fig}
\end{figure}
Let $P_2$ be the Petersen graph with one edge removed 
and
let $P_3$ be the Petersen graph with one vertex removed (see Figure~\ref{fig}). 
We denote by ${\cal P}_2(G)$ the set of subgraphs of $G$ isomorphic to $P_2$. 
We denote by ${\cal P}_{2m}(G)$  the set of subgraphs of $G$ isomorphic to $P_2$ with a triangle inserted  into one vertex or with a $2$-circuit inserted into an edge.
We denote by ${\cal P}_3(G)$ the set of subgraphs of $G$ isomorphic to the graph $P_3$ that 
are not contained in any subgraph from ${\cal P}_2(G)$ or ${\cal P}_{2m}(G)$. 
Let ${\cal P}^*_{2m}(F,G)$ be the set of those graphs $S$ from ${\cal P}_{2m}(G)$ such that $F$ contains a $5$-circuit inside $S$.
Let ${\cal P}(G)={\cal P}_2(G) \cup {\cal P}_3(G) \cup {\cal P}_{2m}(G)$. 
Note that all subgraphs of $G$ from ${\cal P}$ are uncolourable.
It is also not hard to see that subgraphs from ${\cal P}(G)$ are pairwise disjoint.
\begin{lemma}\label{disj}
Let $G$ be a graph not isomorphic to the Petersen graph with $0$, $1$, or $2$ 
triangles or $2$-circuits inserted.
Then the subgraphs from ${\cal P}(G)$ are pairwise disjoint. 
\end{lemma}
\begin{proof}
Let $E_2$ be the set of edges in some $2$-edge-cut separating more than $2$ vertices in $G$.
Let $E_3$ be the set of edges in a $3$-edge-cut separating a graph from ${\cal P}_3(G)$.
One can easily check that edges from $E_2 \cup E_3$ cannot be in the subgraphs from ${\cal P}(G)$.
Suppose for a contradiction that some element $x$, either an edge or a vertex, 
is in both $S_1 \in {\cal P}(G)$ and $S_2 \in {\cal P}(G)$
(we know $x\not \in E_2 \cup E_3$). Let $S$ be the component of
$G-(E_2 \cup E_3)$. Then $S=S_1=S_2$ and the lemma follows.
\end{proof}

We use ${\cal P}$, ${\cal P}_2$, ${\cal P}_3$, ${\cal P}_{2m}$, and ${\cal P}^*_{2m}(F)$ instead of 
${\cal P}(G)$, ${\cal P}_2(G)$, ${\cal P}_3(G)$, ${\cal P}_{2m}(G)$, and ${\cal P}^*_{2m}(F,G)$,
respectively,
when no confusion can occur. 
We will refer to subgraphs from ${\cal P}$ as \emph{special subgraphs}, and $2$- and $3$-circuits inside subgraphs from ${\cal P}_{2m}(G)$ as \emph{special $2$- and $3$-circuits}.

We will prove the following.

\begin{theorem} \label{avoid}
Let $G$ be a $2$-edge-connected cubic graph other than the Petersen graph. 
Then $G$ has a $2$-factor $F$ such that 
\begin{enumerate}
\item $G/F$ is $5$-odd-edge-connected,
\item $I(F,G) \le 1/10 \cdot |V(G)|+ 1/3 \cdot |{\cal P}_2(G)| + 1/10 \cdot |{\cal P}_3(G)| + 2/15 \cdot |{\cal P}^*_{2m}(F,G)|$.
\end{enumerate}
\end{theorem}

To prove Theorem~\ref{avoid} we consider some smallest counterexample $G$ to the theorem and show that it has girth at least $5$, contains no $2$- or $3$-edge-cut that separates a colourable subgraph, and has no subgraphs from ${\cal P}_{2m}$. As the first step we introduce several reduction lemmas.

\subsection{Reductions for the proof of Theorem~\ref{avoid}}
The first three reductions will guarantee that 
each smallest counterexample to Theorem~\ref{avoid} has girth at least $5$.

\begin{lemma}
\label{triangleLemma}
All $2$- and $3$-circuits of each smallest counterexample to Theorem~\ref{avoid} are special.
\end{lemma}
\begin{proof}
Let $G$ be some smallest counterexample to Theorem~\ref{avoid} and suppose that $G$ contains a $3$-circuit $C$ that is not special.
Let $G'=G/C$. Graph $G'$ has a $2$-factor $F'$ satisfying Theorem~\ref{avoid}
and we can easily extend it to a $2$-factor $F$ of $G$.
Since $F'$ is $5$-odd-edge-connected, we do not create a new $5$-circuit in $F$ and $I(F,G) \le I(F',G')$. 
The $3$-circuit is not special, hence we have ${\cal P}_2(G')={\cal P}_2(G)$, but removing the triangle might
create a new subgraph isomorphic to $P_3$ or $P_{2m}$ in $G'$.
However, by Lemma~\ref{disj} only one such subgraph may be created (in case $G'$ is an exceptional graph from Lemma~\ref{disj} one can easily see that $G$ is not a counterexample to Theorem~\ref{avoid}).
Therefore, $|{\cal P}_3(G')| + |{\cal P}^*_{2m}(F',G')| \le |{\cal P}_3(G)| + |{\cal P}^*_{2m}(F,G)| +1$.
For the graph $G'$ we have $V(G')=V(G)-2$.
It follows that
\begin{eqnarray*}
I(F,G) &\le& I(F',G') \\
       &\le& 1/10 \cdot |V(G')|+ 1/3 \cdot |{\cal P}_2(G')| + 1/10 \cdot |{\cal P}_3(G')| + 2/15 \cdot |{\cal P}^*_{2m}(F',G')|\\ 
       &\le& 1/10 \cdot (|V(G)| - 2) + 1/3 \cdot |{\cal P}_2(G)| + 1/10 \cdot |{\cal P}_3(G)| + 2/15 \cdot |{\cal P}^*_{2m}(F,G)| + 2/15\\
       &<& 1/10 \cdot |V(G)| + 1/3 \cdot|{\cal P}_2(G)| + 1/10 \cdot |{\cal P}_3(G)| + 2/15 \cdot |{\cal P}^*_{2m}(F,G)|.
\end{eqnarray*}
The graph $G/F$ has one loop more than $G'/F'$, and since $G'/F'$ is $5$-odd-edge-connected, so is $G/F$, which is a contradiction with $G$ being a counterexample to Theorem~\ref{avoid}.
Finally, if $G'$ is isomorphic to the Petersen graph, then it can be easily verified that $G$ fulfils Theorem~\ref{avoid}, which is again a contradiction.

Similarly, we can show that each smallest counterexample to Theorem~\ref{avoid} contains only special $2$-circuits. We create the graph $G'$ by reducing a $2$-circuit in $G$.
The graph $G/F$ may now differ from $G'/F'$ in two ways: either $G/F$ has one extra loop, 
or it has one edge subdivided with a vertex of degree $2$.
In both cases, $5$-odd-edge-connectivity of $G'/F'$ implies $5$-odd-edge-connectivity of $G/F$, which contradicts the choice of $G$.
\end{proof}

\begin{lemma}
Each smallest counterexample to Theorem~\ref{avoid} contains no $4$-circuit.
\end{lemma}
\begin{proof}
Let $G$ be some smallest counterexample to Theorem~\ref{avoid}. 
According to Lemma~\ref{triangleLemma}, $G$ has no circuits of lengths $2$ and $3$ outside special subgraphs. Suppose that $G$ contains a $4$-circuit $C_4 = v_1v_2v_3v_4v_1$. Note that a special subgraph cannot intersect $C_4$. Let $w_1$, $w_2$, $w_3$, and $w_4$ be the neighbours of $v_1$, $v_2$, $v_3$, and $v_4$ outside of $C_4$, respectively. Since $G$ has no circuits shorter than $4$ outside special subgraphs, such neighbours exist.
We use one of the three following reductions depending on which of the vertices $w_1, w_2,w_3$, and $w_4$ are identical.

\medskip
{\noindent Reduction 1:}
If both $w_1=w_3$ and $w_2=w_4$, then 
we construct $G'$ by deleting the vertices $v_1$, $v_2$, $v_3$, $v_4$, $w_1=w_3$, and $w_2=w_4$ and by adding 
a new edge between the two resulting $2$-valent vertices. It is easy to extend the $2$-factor $F'$ of $G'$ to a
$2$-factor $F$ of $G$ without adding new circuits of length $5$, therefore, $I(F,G) \le I(F',G')$. The construction can create one special subgraph, hence $|{\cal P}_2(G')| + |{\cal P}_3(G')| + |{\cal P}^*_{2m}(F',G')| \le |{\cal P}_2(G)| + |{\cal P}_3(G)| + |{\cal P}^*_{2m}(F,G)| +1$. 
For the graph $G'$ we have $V(G')=V(G)-6$.
It follows that
\begin{eqnarray*}
I(F,G) &\le& I(F',G') \\
       &\le& 1/10 \cdot |V(G')|+ 1/3 \cdot |{\cal P}_2(G')| + 1/10 \cdot |{\cal P}_3(G')| + 2/15 \cdot |{\cal P}^*_{2m}(F',G')|\\ 
       &\le& 1/10 \cdot (|V(G)| - 6) + 1/3 \cdot |{\cal P}_2(G)| + 1/10 \cdot |{\cal P}_3(G)| + 2/15 \cdot |{\cal P}^*_{2m}(F,G)| + 1/3\\
       &<& 1/10 \cdot |V(G)| + 1/3 \cdot|{\cal P}_2(G)| + 1/10 \cdot |{\cal P}_3(G)| + 2/15 \cdot |{\cal P}^*_{2m}(F,G)|.
\end{eqnarray*}

The graph $G/F$ either has extra loops, or it has one extra edge subdivided by a vertex with loops as compared to $G'/F'$.
In both cases, $5$-odd-edge-connectivity of $G'/F'$ implies $5$-odd-edge-connectivity of $G/F$.
\medskip

{\noindent Reduction 2:}
Assume that only one pair of vertices is identical, say $w_1=w_3$. We contract $v_1$, $v_2$, $v_3$, $v_4$, and $w_1=w_3$ into a new vertex to create $G'$. Again, it is easy to extend the $2$-factor $F'$ of $G'$ to a
$2$-factor $F$ of $G$ without adding new circuits of length $5$, therefore, $I(F,G) \le I(F',G')$. The construction can create one special subgraph, hence $|{\cal P}_2(G')| + |{\cal P}_3(G')| + |{\cal P}^*_{2m}(F',G')| \le |{\cal P}_2(G)| + |{\cal P}_3(G)| + |{\cal P}^*_{2m}(F,G)| +1$. For the graph $G'$ we have $V(G')=V(G)-4$ and it follows that $I(F,G)$ satisfies the bound in Theorem~\ref{avoid}.

The graph $G/F$ differs from $G'/F'$ by two extra loops and since $G'/F'$ is $5$-odd-edge-connected, so is $G/F$.
\medskip

{\noindent Reduction 3:}
If all the vertices $w_1$, $w_2$, $w_3$, and $w_4$ are pairwise distinct, then we delete $v_1$, $v_2$, $v_3$, and $v_4$ and either add edges $w_1w_2$ and $w_3w_4$, or $w_1w_4$ and $w_2w_3$ to create $G'$. 
One of these two choices always guarantees that $G'$ is $2$-edge-connected \cite{snarks}. In case one choice creates two new special subgraphs in $G'$, the other choice must create a $2$-edge-connected graph without any new special subgraphs. In such case we choose the latter one. 

Without loss of generality, let $w_1w_2$ and $w_3w_4$ be the edges added to create $G'$.
We extend the $2$-factor $F'$ of $G'$ to a $2$-factor $F$ of $G$. To show that $G/F$ is $5$-odd-edge-connected, we only need to check if the extension of $F'$ into $F$ does not create any new $3$-edge-cuts in $G/F$ since $G'/F'$ is $5$-odd-edge-connected.
Three cases arise depending on which of the two edges $w_1w_2,w_3w_4$ belong to $F'$. 

If $w_1w_2,w_3w_4\not\in F'$, then we set $F=F' \cup \{v_1v_2v_3v_4v_1\}$ and no new $5$-circuit is introduced. We get $G/F$ from $G'/F'$ by subdividing both of the edges corresponding to $w_1w_2,w_3w_4$ in $G'/F'$ and identifying the new vertices. Such operation does not create a new $3$-edge-cut in $G/F$.

If exactly one of the edges $w_1w_2$ and $w_3w_4$ is in $F'$, then
without loss of generality $w_1w_2 \in F'$ and $w_3w_4 \not\in F'$ and we extend $F'$ by the edges 
$\{w_1v_1,v_1v_4,v_4v_3,v_3v_2,v_2w_2\}$ and no new $5$-circuit is introduced. 
To get $G/F$ from $G'/F'$ we subdivide the edge corresponding to $w_3w_4$ in $G'/F'$, add a loop to a vertex corresponding to a circuit containing $w_1w_2$, and we identify the two vertices. This operation does not create a new $3$-edge-cut in $G/F$.

If $w_1w_2,w_3w_4\in F'$, then we extend $F'$ by the edges $v_1v_2$ and $v_3v_4$. No new $5$-circuit is introduced because $w_1w_2$ and $w_3w_4$ cannot be a part of a $3$-circuit in $F'$ as $G'/F'$ is $5$-odd-edge-connected.
We get $G/F$ from $G'/F'$ by either adding a multiple edge between two vertices, or adding two loops to one vertex. This operation does not create a new $3$-edge-cut in $G/F$. Therefore, $G/F$ is $5$-odd-edge-connected.

In all of the three cases we add no new $5$-circuit in $F$ and $I(F,G) \le I(F',G')$. The construction can create one special subgraph, hence $|{\cal P}_2(G')| + |{\cal P}_3(G')| + |{\cal P}^*_{2m}(F',G')| \le |{\cal P}_2(G)| + |{\cal P}_3(G)| + |{\cal P}^*_{2m}(F,G)| +1$. For the graph $G'$ we have $V(G')=V(G)-4$ and it follows that $I(F,G)$ satisfies the bound in Theorem~\ref{avoid}, which is a contradiction.

For each of the above reductions, we can easily check that if $G'$ is isomorphic to the Petersen graph, then $G$ fulfils Theorem~\ref{avoid}, which is a contradiction.
\end{proof}

Up to now we have shown that the only circuits shorter than $5$ contained in some smallest counterexample to Theorem~\ref{avoid} are special $2$- and $3$-circuits, which are by definition contained only in subgraphs from ${\cal P}_{2m}(G)$. As the next step we show that such graph contains no special subgraph from ${\cal P}_{2m}(G)$, and it hence has girth at least $5$.

\begin{lemma}
\label{spLemma}
If $G$ is some  smallest counterexample to Theorem~\ref{avoid}, then ${\cal P}_{2m}(G) = \emptyset$.
\end{lemma}
\begin{proof}
Let $G$ be some smallest counterexample to Theorem~\ref{avoid} and suppose that there exists a subgraph $S$ from ${\cal P}_{2m}(G)$. 
We reduce $S$ to an edge and denote the new graph by $G'$. Let $F'$ be a $2$-factor of $G'$ from Theorem~\ref{avoid}. We extend it to a $2$-factor $F$ of $G$ without creating any circuits of length $2$ and $3$ and as few $5$-circuits as possible. 
Simple case analysis shows that this is always possible without creating more than one $5$-circuit.

First, assume that no new $5$-circuit is created, hence $I(F,G) \le I(F',G')$. 
The graph $G'$ can have one special subgraph more than $G$, thus 
$|{\cal P}_2(G')| + |{\cal P}_3(G')| + |{\cal P}^*_{2m}(F',G')| \le |{\cal P}_2(G)| + |{\cal P}_3(G)| + |{\cal P}^*_{2m}(F,G)| +1$. 
For the graph $G'$ we have $V(G')=V(G)-12$.
It follows that
\begin{eqnarray*}
I(F,G) &\le& I(F',G') \\
       &\le& 1/10 \cdot |V(G')|+ 1/3 \cdot |{\cal P}_2(G')| + 1/10 \cdot |{\cal P}_3(G')| + 2/15 \cdot |{\cal P}^*_{2m}(F',G')|\\ 
       &\le& 1/10 \cdot (|V(G)| - 12) + 1/3 \cdot |{\cal P}_2(G)| + 1/10 \cdot |{\cal P}_3(G)| + 2/15 \cdot |{\cal P}^*_{2m}(F,G)| + 1/3\\
       &<& 1/10 \cdot |V(G)| + 1/3 \cdot|{\cal P}_2(G)| + 1/10 \cdot |{\cal P}_3(G)| + 2/15 \cdot |{\cal P}^*_{2m}(F,G)|.
\end{eqnarray*}

We get $G/F$ from $G'/F'$ by using some of the following operations: subdividing an edge, adding an isolated vertex, adding a $4$-tuple edge between two vertices, and adding a loop.
Therefore, if $G'/F'$ is $5$-odd-edge-connected, then also $G/F$ is $5$-odd-edge-connected.

Assume that one $5$-circuit is created in $F$. Then $I(F,G) \le I(F',G')+1$ and $V(G')=V(G)-12$.
The reduction can create one new special subgraph in $G'$, however, $G'$ does not contain the subgraph $S$ any more, as it was reduced: $S\in{\cal P}^*_{2m}(F,G)$ (there is a $5$-circuit of $F$ in $S$) and $S \not \in{\cal P}^*_{2m}(F',G')$.
Therefore,
\begin{eqnarray*}
I(F,G) &\le& I(F',G') + 1\\
       &\le& 1/10 \cdot |V(G')|+ 1/3 \cdot |{\cal P}_2(G')| + 1/10 \cdot |{\cal P}_3(G')| + 2/15 \cdot |{\cal P}^*_{2m}(F',G')|+1\\ 
       &\le& 1/10 \cdot (|V(G)| - 12) + 1/3 \cdot |{\cal P}_2(G)| + 1/10 \cdot |{\cal P}_3(G)| + 2/15 \cdot (|{\cal P}^*_{2m}(F,G)|-1) + 1/3+1\\
       &\le& 1/10 \cdot |V(G)| + 1/3 \cdot|{\cal P}_2(G)| + 1/10 \cdot |{\cal P}_3(G)| + 2/15 \cdot |{\cal P}^*_{2m}(F,G)|.
\end{eqnarray*}
It follows that $I(F,G)$ satisfies the bound in Theorem~\ref{avoid}.

As in the previous case the graph $G/F$ can be obtained from $G'/F'$ by using some of the following operations: subdividing an edge, adding an isolated vertex, adding a $4$-tuple edge between two vertices, and adding a loop.
Therefore, if $G'/F'$ is $5$-odd-edge-connected, then also $G/F$ is $5$-odd-edge-connected.
We can easily check that if $G'$ is isomorphic to the Petersen graph, then $G$ fulfils Theorem~\ref{avoid}. In any case, we get a contradiction.
\end{proof}

The next two reductions guarantee that any $2$- or $3$-edge-cut separates an uncolourable subgraph in each smallest counterexample to Theorem~\ref{avoid}.

\begin{lemma}
Each smallest counterexample to Theorem~\ref{avoid} does not contain any $2$-edge-cut separating a colourable subgraph.
\label{lemma2edgeCut2}
\end{lemma}
\begin{proof}
Let $G$ be one of the smallest counterexamples to Theorem~\ref{avoid} and suppose that there exists a $2$-edge-cut separating a colourable subgraph in $G$.
Let $v_1v_2$ and $w_1w_2$ be the cut-edges such that
$v_1$ and $w_1$ are in the colourable component $H$ of 
$G-\{v_1v_2,w_1w_2\}$. Clearly, $v_1$, $v_2$, $w_1$, and $w_2$ are all distinct 
otherwise $G$ has a bridge.

We create two components $G_1$ and $G_2$ by deleting the edges $v_1v_2$ and $w_1w_2$ in $G$ and adding two new edges $v_1w_1$ and $v_2w_2$.
Let $G_1$ be the component containing $v_1$ (the colourable one).
We fix a proper colouring $c$ of $G_1$ with colours $\{1, 2, 3\}$ such that one edge between vertices $v_1$ and $w_1$ gets the colour $1$ and the other one (if there is such edge) gets the colour $3$.
The graph $G_2$ has a $2$-factor $F'$ satisfying Theorem~\ref{avoid} and we extend it to a $2$-factor $F$ of $G$ as follows. 

If $v_2w_2\not\in F'$, then we extend $F'$ by the
edges of $G_1$ coloured by $2$ and $3$ and no new $5$-circuits are created in $F$.
We have to show that $G/F$ is $5$-odd-edge-connected. 
We know that the graph $G_2/F'$ is $5$-odd-edge-connected and that $G_1/(F-F')$ is Eulerian.
We show that $G/F$ has no $3$-edge-cut. Assume for the contrary that there is a $3$-edge-cut in
$G/F$ that separates $G/F$ into two components $A$ and $B$. None of these subgraphs can be completely within
$G_2/F'$ as this would imply a $3$-edge-cut in $G_2/F'$. Similarly,
None of these subgraphs can be completely within
$G_1/(F-F')$ as this would imply a $3$-edge-cut in $G_1/(F-F')$ which is, again, not possible because 
$G_1/(F-F')$ is Eulerian. Therefore, both $A$ and $B$ contain vertices both from  $G_2/F'$ and $G_1/(F-F')$.
Without loss of generality let $v_1v_2\in A$ and $w_1w_2\in B$. From the three cut-edges separating $A$ from $B$
either one or two are in $G_1/(F-F')$. However, none of these two cases is possible: 
if one cut-edge is in $G_1/(F-F')$, then the remaining two edges together with $w_1w_2$ form a $3$-edge-cut of $G_2/F'$
and
if two cut-edges separating $A$ from $B$ are in $G_1/(F-F')$, then these remaining two edges 
together with $v_1v_2$ form a $3$-edge-cut of $G_1/(F-F')$. Therefore, $G/F$ is $5$-odd-edge-connected.

If $v_2w_2\in F'$, then we put the following edges to $F$:
edges from $F'$, 
edges coloured by $1$ and $2$ in $G_1$, and 
the edges $v_1v_2$ and $w_1w_2$. 
The circuits of $F$ inside $G_1$ are even and since $v_1$ and $w_1$ are not neighbours on a circuit in $F$ 
(either there is no edge $v_1w_1$, or it is coloured by $3$ and hence it is not part of $F$). 
The only new $5$-circuit that can be created is
$vv_1v_2w_2w_1vv$ where $v$ is a common neighbour of $v_1$ and $w_1$. However, this is not possible, because one of the edges $v_1v$ and  $w_1v$ is coloured by $2$ and the other one by $3$ in $G_1$. Hence, no new $5$-circuits, besides the $5$-circuits of $F'$, are introduced in $F$.
No new $3$-edge-cut can be created in $G/F$
as $G/F$ is created from $G_2/F'$ and an Eulerian graph by identification of two vertices.

As no new $5$-circuits are created in $F$, we have $I(F,G) \le I(F',G_2)$. 
The graph $G_2$ can have one special subgraph more than $G$, thus 
$|{\cal P}_2(G_2)| + |{\cal P}_3(G_2)| + |{\cal P}^*_{2m}(F',G_2)| \le |{\cal P}_2(G)| + |{\cal P}_3(G)| + |{\cal P}^*_{2m}(F,G)| +~1$.
For the graph $G'$ we have $V(G')\le V(G)-4$ and it follows that $I(F,G)$ satisfies the bound in Theorem~\ref{avoid}, which is a contradiction.
We can easily check that if $G_2$ is isomorphic to the Petersen graph, then $G$ fulfils Theorem~\ref{avoid}.
\end{proof}

\begin{lemma}
Each smallest counterexample to Theorem~\ref{avoid} does not contain any non-trivial $3$-edge-cut separating a colourable subgraph.
\label{lemma3edgeCut2}
\end{lemma}
\begin{proof}
Let $G$ be some smallest counterexample to Theorem~\ref{avoid}.
We can assume by the above lemmas that $G$ has girth at least $5$ and no $2$-edge-cut separates a colourable subgraph in $G$.
Suppose that there exists a $3$-edge-cut separating a colourable subgraph in $G$. We choose a non-trivial $3$-edge-cut that separates the smallest colourable subgraph. 
Let $v_1v_2$, $w_1w_2$, $x_1x_2$ be the cut-edges such that the vertices $v_1$, $w_1$, and $x_1$ are in the colourable subgraph. The vertices $v_1$, $v_2$, $w_1$, $w_2$, $x_1$, and $x_2$ are pairwise distinct otherwise there is a $2$-edge-cut in the graph.
Moreover, no two vertices of $v_1$, $w_1$, and $x_1$ are neighbours as this either contradicts the choice of
the $3$-edge-cut or the fact that $G$ has no triangles.

We create two components $G_1$ and $G_2$ by adding two new vertices $y_1$ and $y_2$, deleting the edges $v_1v_2$, $w_1w_2$, and $x_1x_2$
and adding new edges $v_1y_1$, $w_1y_1$, $x_1y_1$, $y_2v_2$, $y_2w_2$, and $y_2x_2$.
Let $F'$ be a $2$-factor of $G_2$ that satisfies Theorem~\ref{avoid} and let $c$ be a $3$-edge-colouring of $G_1$. We can easily extend $F'$ to a $2$-factor $F$ of $G$. Without loss of generality, suppose that $y_2v_2$ and $y_2w_2$ are contained in $F'$. Then we add all edges in $G_1$ coloured by $c(y_1v_1)$ and $c(y_1w_1)$ to $F$. The circuits of $F$ inside $G_1$ are even. The only new $5$-circuit that can be created is $vv_1v_2w_2w_1v$ where $v$ is a common neighbour of $v_1$ and $w_1$. This is not possible, because $G_2$ does not contain a triangle in $F'$. Hence, no new $5$-circuits are created in $F$ and $I(F,G) \le I(F',G_2)$.
The graph $G_2$ can have one special subgraph more than $G$, thus 
$|{\cal P}_2(G_2)| + |{\cal P}_3(G_2)| + |{\cal P}^*_{2m}(F',G_2)| \le |{\cal P}_2(G)| + |{\cal P}_3(G)| + |{\cal P}^*_{2m}(F,G)| +~1$. For the graph $G'$ we have $V(G')\le V(G)-5$ and it follows that $I(F,G)$ satisfies the bound in Theorem~\ref{avoid}.

We can obtain the graph $G/F$ from $G_2/F'$ as follows. 
Let $e_1$ be the edge of $G_1/(F-F')$ that corresponds to $x_1y_1$.
Let $e_2$ be the edge of $G_2/F'$ that corresponds to $x_2y_2$.
Let $e$ be the edge of $G/F$ that corresponds to $x_1x_2$.
Let $c$ we a vertex of $G/F$ corresponding to circuit of $F$ that contains edges $v_1v_2$
and $w_1w_2$.
The graph $G_1/(F-F')$ is Eulerian and it can be decomposed into circuits. 
Let $C$ be a circuit of $G_1/(F-F')$ that contains $e_1$.
In $G/F$, the circuit $C$ transforms to a path starting at $c$ and ending in $e$.
The edge $e_2$ transforms in $G/F$ into $e$.
Thus we can add edges from $C$ to $G_2/F'$ by subdividing $e_2$
with vertices of degree $2$. The remaining circuits of $G_1/(F-F')$ are preserved in 
$G/F$, and we may add them to $G_2/F'$ directly. In the process we only subdivide edges 
with vertices of degree $2$ and add new disjoint circuits.
Thus no new $3$-edge-cut can be created.
We can easily check that if $G_2$ is isomorphic to the Petersen graph, then $G$ fulfils Theorem~\ref{avoid}.
\end{proof}

\subsection{Proof of Theorem~\ref{avoid}}

Assume that $G$ is some smallest counterexample to Theorem~\ref{avoid}. According to Lemmas~\ref{triangleLemma}-\ref{lemma3edgeCut2}, $G$ has girth at least $5$, every $2$- and $3$-edge-cut separates an uncolourable subgraph.
Recall that under these conditions ${\cal P} = {\cal P}_2 \cup {\cal P}_3$.

Let ${\cal C}_5(G)$, or ${\cal C}_5$ if no confusion can occur, be the set of $5$-circuits of $G$ that do not intersect any subgraph from ${\cal P}$. 
Let $M$ be a perfect matching of $G$, let $F_M$ be the complementary $2$-factor, and let $I(F_M,G)$ be the number of $5$-circuits in $F_M$.

For each $C\in{\cal C}_5$ we define $I(C,M)$ as follows:
$I(C,M)=1$ if $C\in F_M$ and $I(C,M)=0$ otherwise.
For each $S\in {\cal P}$ and each $C\in F_M$ we define $I(C,S,M)$ as follows:
$I(C,S,M)=1$ if $C$ is a $5$-circuit that intersects $S$ 
and $I(C,S,M)=0$ otherwise.
Note that a $5$-circuit cannot intersect two subgraphs from ${\cal P}$.
Let $I(S,M)=\sum_{C\in F_M} I(C,S,M)$. By definition, we can express $I(F_M,G)$ as
\begin{eqnarray}
I(F_M,G)=\sum_{C\in {\cal C}_5} I(C,M) + \sum_{S\in{\cal P}_2} I(S,M) + \sum_{S\in{\cal P}_3} I(S,M). 
\label{eqn1}
\end{eqnarray}

For each $S\in{\cal P}_2$, let $e_S$ be one arbitrary edge from $\delta(S)$.
We define a linear function on ${\cal M}(G)$ as
\begin{eqnarray}
f({\bf p})=\left( 1/4 \cdot  \sum_{C \in {\cal C}_5} \sum_{e\in \delta(C)} p_e  \right) + \sum_{S\in {\cal P}_2} p_{e_S}.
\label{eqn2}
\end{eqnarray}

\begin{lemma}
Let $G$ be some smallest counterexample to Theorem~\ref{avoid}.
Then $G$ has a $2$-factor $F_M$ such that 
\begin{enumerate}
\item $G/F_M$ is $5$-odd-edge-connected,
\item $I(F_M,G) \le  1/6 \cdot |{\cal C}_5|  +   4/3 \cdot |{\cal P}_2| + |{\cal P}_3|$.
\end{enumerate}
\label{lemmaIFMG}
\end{lemma}
\begin{proof}
Any perfect matching of a cubic bridgeless graph intersects each $3$-edge-cut of $G$ in exactly one or three edges, therefore, the sum of weights of edges in a $3$-edge-cut is either $1$ or $3$
in each perfect matching.
The point ${\bf p_0} =(1/3, 1/3, 1/3, \dots, 1/3 )$ lies in ${\cal M}(G)$. 
Therefore, the point ${\bf p_0}$ is a convex combination of some perfect matchings. 
As in each $3$-edge-cut the sum of weights of ${\bf p_0}$ is $1$, the sum must be $1$
for each perfect matching in the convex combination. 
From such perfect matchings we choose $M$ for which $f(M)$ is minimum, that is $f(M)\le f({\bf p})$ for all ${\bf p} \in {\cal M}(G)$. Since $M$ intersects every $3$-edge-cut in exactly one edge, the graph $G/F_M$ is $5$-odd-edge-connected, and hence the first part of the lemma holds.

We now prove the second part of the lemma. Since $M$ is a matching in a convex combination of matchings yielding $(1/3,1/3,\ldots,1/3)$, and among those $M$ is a perfect matching such that $f(M)$ is minimal, we have $f(M) \le f(1/3,1/3,\ldots,1/3)$ and thus
\begin{eqnarray}
f(M) \le  5/12 \cdot  |{\cal C}_5|  + 1/3 \cdot |{\cal P}_2|.
\label{eqn3}
\end{eqnarray}
Let ${\bf p}$ be the characteristic vector of $M$.
First let us consider the value $I(C,M)$ for a circuit $C\in {\cal C}_5$. 
If $C\in F_M$, then all the edges from $\delta(C)$ belong to $M$, that is $\sum_{e\in \delta(C)} p_e=5$, and by the definition of $I$, we have $I(C,M)=1$. 
If $C\not\in F_M$, then $\sum_{e\in \delta(C)} p_e\ge 1$ as $C$ has an odd number of vertices
and by definition $I(C,M)=0$. We can bound $I(C,M)$ by 
\begin{eqnarray}
I(C,M)\le -1/4+1/4 \cdot \sum_{e\in \delta(C)} p_e.
\label{eqn4}
\end{eqnarray}
Summing (\ref{eqn4}) over all circuits $C\in {\cal C}_5$ we get
$$
\sum_{C\in {\cal C}_5} I(C,M) \le 
\left(1/4 \cdot \sum_{C \in {\cal C}_5}   \sum_{e\in \delta(C)} p_e \right)- 1/4 \cdot |{\cal C}_5|.
$$

Consider the value of $I(S,M)$. The circuits of each $2$-factor intersecting $S$
can be extended to a $2$-factor of the Petersen graph, which consists of precisely 
two $5$-circuits. Thus we can calculate the number of $5$-circuits intersecting $S$ when the value $p_{e_S}$ is known.
Assume that $S \in {\cal P}_2$. If $p_{e_S}=1$, then both edges from $\delta(S)$ belong to $M$, which implies two $5$-circuits in $S\cap F_M$ and we have $I(S,M)=2$. If $p_{e_S}=0$, then $I(S,M)=1$. In both cases we have $I(S,M) \le p_{e_S} + 1$ (note that equality holds here but we do not need it) 
and thus
$$
\sum_{S\in{\cal P}_2} I(S,M) \le \left(\sum_{S\in {\cal P}_2} p_{e_S}\right)+|{\cal P}_2|.
$$
Assume that $S \in {\cal P}_{3}$. There is at most one $5$-circuit from $F_M$ intersecting $S$, otherwise $S$ could be extended to a subgraph isomorphic to $S_2$, which is not possible by the definition. Therefore, $I(S,M) \leq 1$ and
$$
\sum_{S\in{\cal P}_3} I(S,M)\le |{\cal P}_3|.
$$
Combining (\ref{eqn1})-(\ref{eqn3}) and the three inequalities bounding the summands of $I(F_M,G)$, we get that
\[
I(F_M,G) \leq f({\bf p}) -1/4 \cdot |{\cal C}_5| + |{\cal P}_2| + |{\cal P}_3| \leq 1/6 \cdot |{\cal C}_5|  + 4/3 \cdot |{\cal P}_2| +|{\cal P}_3|.\qedhere 
\]
\end{proof}

Let $V_{P_2}$ and $V_{P_3}$ be the vertices of subgraphs from ${\cal P}_2$, ${\cal P}_3$, respectively. 
Our next aim is to bound the number of vertices of $G$ outside special subgraphs in terms of the number of $5$-circuits in $G$. The lemma corresponds to Lemma 11 in \cite{CL}. As the proof of the lemma is almost identical, we only outline the main ideas.
\begin{lemma}
Let $G$ be some smallest counterexample to Theorem~\ref{avoid}.
Then $|V(G)-V_{P_2}-V_{P_3}| \ge 5/3 \cdot |{\cal C}_5|$.
\label{lemmaVert}
\end{lemma}
\begin{proof}
Recall that the graph $G$ has girth at least $5$, every $2$- and $3$-edge-cut in $G$ separates 
two uncolourable subgraphs of $G$, and ${\cal P} = {\cal P}_2 \cup {\cal P}_3$.

For $2\le k\le 4$ let $V_k$ be the set of vertices outside ${\cal P}$ that are in exactly $k$ $5$-circuits, and let $V_1$ be the set of vertices outside ${\cal P}$ that are in at most one $5$-circuit. 
Let $n = |V(G)-V_{P_2}-V_{P_3}|$. We determine the number of pairs $(v,C)$ where $C$ is a $5$-circuit and $v$ is a vertex of $C$ outside ${\cal P}$. The number of pairs is clearly at most $4|V_4|+3|V_3|+2|V_2|+|V_1|$.
We will show later that $|V_4| \le |V_2 \cup V_1|$. 
Using this inequality,
$4|V_4|+3|V_3|+2|V_2|+|V_1|\leq 3(|V_4|+|V_3|+|V_2|+|V_1|) - |V_1| \le 3n$.
On the other hand, we have exactly $5 \cdot |{\cal C}_5|$ such pairs. Therefore, $3 n \ge 5|{\cal C}_5|$, and $n \ge 5/3|{\cal C}_5|$.

Now it is sufficient to prove that $|V_4| \le |V_2 \cup V_1|$. We do this by finding an injective function
from $V_4$ to the set $\{(v',C) \ : \ v'\in V_2 \cup V_1, v' \in C, C \in {\cal C}_5\}$.
Since there are at most two possible circuits $C$ for each vertex $v'$, the existence of this function implies the desired inequality. 

Let $v\in V_4$. By case analysis one can show that except for several small graphs up to sixteen vertices, for which the lemma is true,
there are up to symmetries only three possible neighbourhoods of $v$: $S_{242}$, $S_{323a}$, and $S_{323b}$
(see Figure~\ref{fig2}). 
\begin{figure}[htp]
\center
\includegraphics{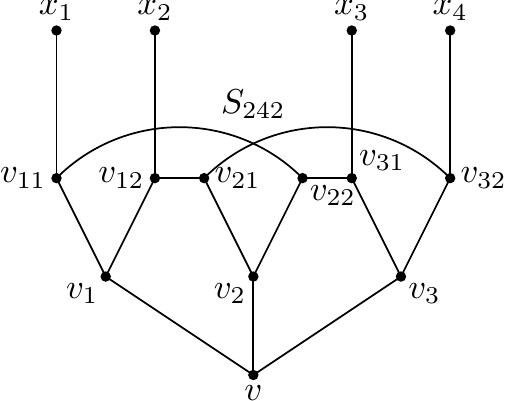} \\
\includegraphics{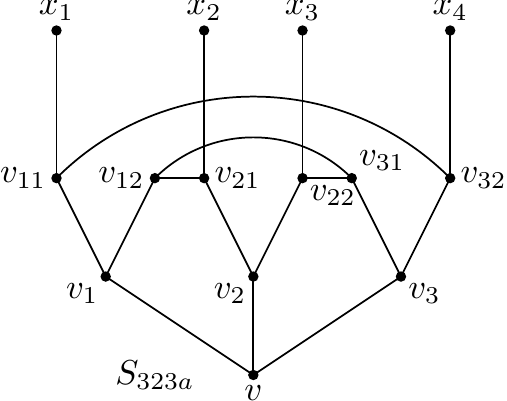} \ \ \ \ \ \ \ \ \ \includegraphics{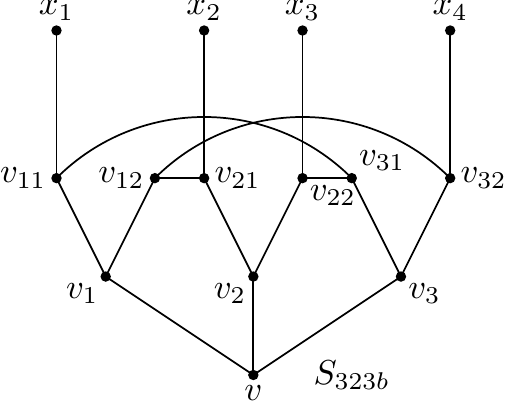}\\
\caption{Possible surroundings of a vertex contained in four $5$-circuits \cite{CL}.}
\label{fig2}
\end{figure} 
For the configuration $S_{242}$, 
we can assign $(v_{11},vv_1v_{11}v_{22}v_2)$ 
and $(v_{31},vv_3v_{31}v_{22}v_2)$ to $v$, 
while we simultaneously assign $(v_{12},vv_1v_{12}v_{21}v_2)$ 
and $(v_{32},vv_3v_{31}v_{22}v_2)$ to $v_2$ (which is also in $V_4$).
For the configuration $S_{323a}$, 
we can assign $(v_{11},vv_1v_{11}v_{32}v_3)$ 
and $(v_{32},vv_1v_{11}v_{32}v_3)$ to $v$.
For the configuration $S_{323b}$, 
we can assign $(v_{11},vv_1v_{11}v_{31}v_3)$ 
and $(v_{32},vv_1v_{12}v_{32}v_3)$ to $v$.
One can show that the assigned vertices are in $V_2 \cup V_1$ and that in assigned circuits vertices other than $v$ 
(and $v_2$ for the configuration $S_{242}$) are not in $V_4$. Therefore, this defines the desired injective function.
\end{proof}

\begin{proof}[Proof of Theorem~\ref{avoid}]
Let $G$ be some smallest counterexample to Theorem~\ref{thmmain}. By Lemmas~\ref{triangleLemma}-\ref{lemma3edgeCut2}, we can assume that $G$ has girth at least $5$, every $2$- or $3$-edge-cut separates an uncolourable subgraph of $G$, and ${\cal P}_{2m}=\emptyset$. Combining Lemmas~\ref{lemmaIFMG} and~\ref{lemmaVert}
we get that $G$ has a $2$-factor $F$ such that $G/F$ is $5$-odd-edge-connected and 
\begin{eqnarray*}
I(F,G) &\le&  1/6 |{\cal C}_5|  +  4/3 |{\cal P}_2| + |{\cal P}_3| \le \\
&\le&  1/10 |V(G)-V_{P_2}-V_{P_3}| + 4/30 |V_{P_2}| + 1/9 |V_{P_3}|= \\
&=& 1/10 |V(G)| + 1/30 |V_{P_2}| + 1/90 |V_{P_3}|= \\
&=& 1/10 |V(G)|+ 1/3 |{\cal P}_2|+1/10 |{\cal P}_3|
\end{eqnarray*}
which proves the theorem.
\end{proof}

We easily get the following corollaries from Theorem~\ref{avoid}.
\begin{corollary}
\label{cor3EdgeConn}
Let $G$ be a cyclically $3$-edge-connected cubic graph on $n$ vertices other than the Petersen graph. 
Then $G$ has a $2$-factor with at most $1/9 \cdot n$ circuits of length $5$. 
\end{corollary}
\begin{corollary}
\label{cor4EdgeConn}
Let $G$ be a cyclically $4$-edge-connected cubic graph on $n$ vertices other than the Petersen graph. 
Then $G$ has a $2$-factor with at most $1/10 \cdot n$ circuits of length $5$. 
\end{corollary}

\section{Short cycle covers of cubic graphs}
In this section we will use the methods from \cite{kaiser} to construct a cycle cover of a bridgeless cubic graph and we prove a new upper bound on the length of this cover.
\begin{theorem}
\label{thmmain}
Every bridgeless cubic graph with $m$ edges has a cycle cover of length at most $1.6m$.
\end{theorem}

To prove this theorem we take the $2$-factor $F$ from Theorem~\ref{avoid} and colour the edges outside of $F$ using three colours. This having as a basis, we create two cycle covers of $G$ and we bound their total lengths in terms of specific circuit lengths in $F$. As the last step, we make a convex combination of the two bounds, which gives us the required bound for the shortest cycle cover.

\subsection{Rainbow $2$-factor}
A \emph{rainbow $2$-factor} in $G$ is a $2$-factor of $G$ together with the colouring of edges from $G/F$ with three colours R, G, and B such that in $G/F$ the numbers of edges of each colour incident to a vertex have the same parity. 
By the result of Jaeger \cite{jaeger} we know that a $5$-odd-connected graph has 
a nowhere-zero $\mathbb{Z}_2^{\; 2}$-flow.
As $G/F$ is $5$-odd-connected, we have a nowhere-zero $\mathbb{Z}_2^{\; 2}$-flow on $G/F$ and we can map the elements of $\mathbb{Z}_2^{\; 2}$
to the set of colours R, G, B to obtain the desired colouring. Hence the $2$-factor from Theorem~\ref{avoid} can be extended into a rainbow $2$-factor.

Let $C=v_1v_2\ldots v_kv_1$ be a circuit of a rainbow $2$-factor $F$ and $P_1,P_2,\ldots, P_k$ be the colours of the non-circuit edges incident with $v_1,v_2,\ldots,v_k$, respectively. We say that $C$ has \emph{type} $P_1P_2\ldots P_k$. We consider two types to be the same if we can obtain one of them by some rotations and reflections from the other one. We do not allow colour permutation, therefore, RRGG and RGGR are the same type but RRGG and RRBB are not.
 Similarly, we say that $C$ has \emph{pattern} $P_1P_2\ldots P_k$.
We consider two patterns to be the same if we can obtain one of them by some rotations, reflections, and colour permutations from the other one, therefore, RRGG and BRRB are the same patterns.
We impose several additional constraints on patterns of circuits in $2$-factor $F$ from Theorem~\ref{avoid}.

Instead of taking an arbitrary nowhere-zero $\mathbb{Z}_2^{\; 2}$-flow on $G/F$, we use splitting lemmas to reduce the number of available patterns for short circuits.
Let $H$ be a graph, let $v$ be a vertex of degree at least four in $H$, and $v_1$ and $v_2$ be two of his neighbours. We denote by $H.v_1vv_2$ a graph obtained from $H$ by removing the edges $vv_1,vv_2$ and adding a path of length two between $v_1$ and $v_2$. We call this operation \emph{splitting} the vertex $v$. The following lemmas from \cite{kaiser} states that certain splitting preserves the property of a graph to be $5$-odd-connected.
\begin{lemma}
Let $H$ be a $5$-odd-connected graph, and let $v$ be a vertex of degree four and $v_1,v_2,v_3$, and $v_4$ its four neighbours. Then the graph $H.v_1vv_2$ or $H.v_2vv_3$ is also $5$-odd-connected graph.
\label{lemma4vertex}
\end{lemma}
\begin{lemma}
Let $H$ be a $5$-odd-connected graph, and let $v$ be a vertex of degree six and $v_1,\ldots,v_6$ its neighbours. At least one of the graphs $H.v_1vv_2$, $H.v_2vv_3$, and $H.v_3vv_4$ is also $5$-odd-connected.
\label{lemma6vertex}
\end{lemma}
\begin{lemma}
Let $H$ be a $5$-odd-connected graph, and let $v$ be a vertex of degree $d\ge 6$ and $v_1,\ldots,v_d$ its neighbours. 
At least one of the graphs $H.v_ivv_{i+1}$, where $i\in \{1, 2, 3, 4, 5\}$, is also $5$-odd-connected.
\label{lemma8vertex}
\end{lemma}
Doing splitting operations in certain manner puts some constraints on  patterns of circuits in a $2$-factor. Suppose that $G$ is a bridgeless cubic graph with a $2$-factor $F$ such that $G/F$ is $5$-odd-connected. After applying splitting operations on vertices of $G/F$ we can find a nowhere-zero $\mathbb{Z}_2^{\; 2}$-flow of $G/F$ where the edges that were split, receive the same value.

Let $F$ be a rainbow $2$-factor.
For a circuit $C \in F$ we define an \emph{improvement function} $P(C,F)$, which will help us to bound the length of cycle covers. (As we will see later, certain colour types guarantee shorter cycle cover, hence the value of $P(C,F)$ expresses by how much the general bound is locally improved on one circuit $C$.)
\begin{center}
$P(C,F) =
\begin{cases}
1 & \text{if $C$ has type RR} \\
3 & \text{if $C$ has type GG} \\
3 & \text{if $C$ has type BB} \\
1 & \text{if $|C|=4$} \\
1 & \text{if $|C|=6$ and type of $C$ contains only one colour}\\
0 & \text{if $|C|=6$ and type of $C$ contains exactly two colours: R and G} \\
1 & \text{if $|C|=6$ and type of $C$ contains exactly two colours: R and B} \\
2 & \text{if $|C|=6$ and type of $C$ contains exactly two colours: B and G} \\
1 & \text{if $|C|=6$ and type of $C$ contains  all three colours} \\
0 & \text{otherwise}
\end{cases}$
\end{center}

Let $P(F)$ be the sum of $P(C,F)$ for all circuits $C\in F$. Let $d_i$ be the number of circuits of length $i$ in $F$. We extend the result from \cite{kaiser} by constructing a rainbow $2$-factor of the following properties.
\begin{lemma}
\label{patternLemma}
Every cubic bridgeless graph $G$ other than the Petersen graph has rainbow $2$-factors $F_1$ and $F_2$ such that
\begin{enumerate}
\item[1a.] the number of $5$-circuits in $F_1$ is at most $1/10\cdot |V| + 1/3|{\cal P}_2| + 1/10|{\cal P}_3|$,
\item[1b.]the number of $5$-circuits in $F_2$ is at most $1/6$ of all $5$-circuits in $G$,
\item[2.] $G/F$ is $5$-odd-connected, for $F\in \{F_1, F_2\}$,
\item[3.] every circuit of length four has pattern RRRR or RRGG, 
\item[4.] every circuit of length six has pattern RRRRRR, RRRRGG, RRGRRG, RRGGBB, or RRGBBG,
\item[5.] every circuit of length eight has one of the following patterns: RRRRRRRR, RRRRRRGG, RRRRGGGG, RRRRGGBB, RRGGRRGG, RRGGRRBB, RRRRGRRG, RRRRGBBG, RRGGRGGR, RRGGRBBR, RRGGBRRB, RRRRGRGR, RRRGBGBR, RRGRGRGG, RRGRBRBG, and RRGGBGBG,
\item[6.] for each $P \in {\cal P}_2 \cup {\cal P}_3$ at least one $5$-circuit has pattern  RRRGB,
\item[7.] $P(F) \ge 7/3 \cdot d_2 + d_4 + d_6$, for $F\in \{F_1, F_2\}$.
\end{enumerate}

\end{lemma}

\begin{proof}
Let $F_1$ be the $2$-factor from Theorem~\ref{avoid}.
For this $2$-factor statements 1a and 2 hold.
Also, one can easily find a $2$-factor, such that 1b and 2 hold, using the proof of Lemma~\ref{lemmaIFMG} ignoring special subgraphs (see also Proposition~5 of \cite{LMMS}). For all other statements, the remaining of the proof is the same for both $2$-factors. We denote by $F$ an arbitrary $2$-factor of those two.

We repeatedly use Lemmas~\ref{lemma4vertex}-\ref{lemma8vertex} in the graph $G/F$ to obtain a $5$-odd-connected graph with no vertices of degree $4$, $6$, or $8$. Such graph has a nowhere-zero $\mathbb{Z}_2^{\; 2}$-flow \cite{jaeger}, which gives a rainbow colouring of the edges in $G/F$, where each pair of edges that are split from a vertex gets the same colour. Therefore, each $4$-circuit in $F$ must have either the pattern RRRR or RRGG and the third part of the lemma is satisfied. 

Consider a circuit of length $6$ in $F$ with incident edges $e_1,\ldots,e_6$. Without loss of generality, by Lemma~\ref{lemma6vertex} let $e_1$ and $e_2$ be the first pair of edges split from the vertex. The second pair can be either $e_3$ and $e_4$ or $e_4$ and $e_5$ up to symmetry by Lemma~\ref{lemma4vertex}. Therefore, only the patterns of the form $P_1P_1P_2P_2P_3P_3$ and $P_1P_1P_3P_2P_2P_3$ can be created where $P_1,P_2,P_3 \in \{R,G,B\}$, and the fourth statement holds. Similar process can be used for the $8$-circuits of $F$ from the fifth statement of the lemma. We refer the reader to Lemma 14 from \cite{kaiser} where the proof of this statement can be found with all details.

Let us prove the part 6.
Consider a $5$-circuit inside a special subgraph $P$ from ${\cal P}_2 \cup {\cal P}_3$. 
There are three possibilities how $F$ can intersect $P$ as shown on Figure~\ref{flema}. 
\begin{figure}[htp!]
\center
\includegraphics[scale=1.3]{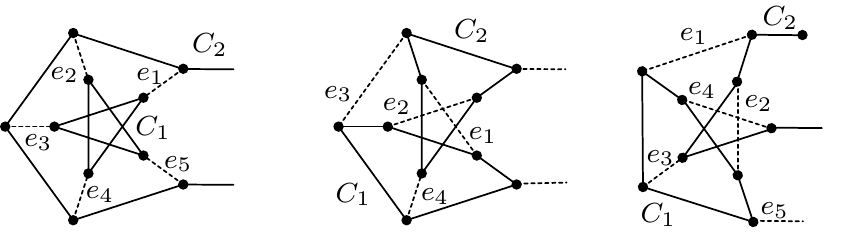}
\caption{$2$-factor in $P_1$ and $P_3$}
\label{flema}
\end{figure}
In the first case, the $2$-factor intersects $P$ in one $5$-circuit $C_1$ and a circuit $C_2$ with $|C_2| \in \{7, 9, 10, 11, 12, \dots\}$ (for lengths $6$ and $8$ there would be either a bridge or a $3$-cut in $G/F$). Let $v_C$ be a vertex in $G/F$ corresponding to the circuit $C$ in $G$. The vertex $v_{C_1}$ is connected to $v_{C_2}$ by $5$ edges $e_1, \ldots, e_5$. For a given rainbow colouring, we permute the colours of the edges $e_1, \ldots, e_5$ in such a way that $C_1$ gets the prescribed pattern. We can do this because there is no prescribed pattern for $C_2$.

In the second case, the circuits $C_1$ and $C_2$ are both $5$-circuits. The vertices $v_{C_1}$ and $v_{C_2}$ are connected by $4$ edges $e_1, \ldots, e_4$. We permute the colours of these edges so that $C_1$ gets the prescribed pattern. No other required patterns in the graph are effected.

In the third case, we have one $5$-circuit $C_1$ and a circuit $C_2$ with $|C_2|\ge 7$ (if $C_2$ was a $6$-circuit, then there would be a $3$-cut in $G/F$). The vertices $v_{C_1}$ and $v_{C_2}$ are connected by $4$ edges $e_1, \ldots, e_4$. If $|C_2|\neq 8$, then we permute the colours of edges $e_1, \ldots, e_4$ so that $C_1$ gets the prescribed pattern without effecting other required patterns in the graph.
However, if $|C_2|=8$, then we need to consider the pattern of $C_2$ since it has to satisfy the part $5$ of this lemma. Let $f_1f_2f_3f_4e_1e_2e_3e_4$ denote the boundary edges of $C_2$. 
The edges $e_5, f_1, \dots, f_4$ form a $5$-edge-cut in $G/F$, therefore, each colour has to be used one or three times on these edges in the rainbow colouring. Without loss of generality suppose that the colour R is used three times on $\{e_5, f_1, \dots, f_4\}$. There are, up to symmetry and colour permutations, five possible colourings of $f_1, f_2, f_3, f_4$ (the colour of $e_5$ is determined by these four colours). For each of these colourings, we can choose the colours of $e_1, \dots, e_4$ in such a way that the circuits have required patterns. The possible colourings of $f_1, f_2, f_3, f_4$ are listed in the table below, where one can check the patterns of $C_1$ (in column 4) and the patterns of $C_2$ (we get the pattern by  merging columns 1 and 3, moreover, in column 5 the same pattern is permuted
by a colour permutation on row 6 and then it is rotated and/or reversed to match some pattern in statement~5). 
\begin{center}
\begin{tabular}{cccccc}
$f_1f_2f_3f_4$ &  $e_5$ & $e_1e_2e_3e_4$ & $e_5e_3e_1e_4e_2$ & pattern of $C_2$ & colour permutation \\\hline
RRRG & B & BGBR & BBBRG & RRRGBGBR & \text{id}\\
RRGR & B & BRBG & BBBGR & RRGRBRBG & \text{id}\\
RRGB & R & RBRG & RRRGB & RRGRBRBG & \text{id}\\
RGBR & R & BRGR & RGBRR & RRGRBRBG & \text{id}\\
RGRB & R & GGBG & RBGGG & RRGRBRBG & (RGB)
\end{tabular}
\end{center}

Now we prove the last part of the lemma.
Let $F_1,\ldots, F_6$ be six rainbow $2$-factors obtained by all permutations of the colours R, G, and B in $G/F$. By definition of $P(F)$ we have 
$\sum_{i=1}^6 P(F_i) = 14d_2 + 6d_4 + 6d_6$. Therefore, at least one of the rainbow $2$-factors satisfies $P(F) \ge 7/3 \cdot d_2 + d_4 + d_6$. The other conditions of the lemma do not change with permutation of colours, hence we choose this $2$-factor.
\end{proof}

\subsection{The first cover}
Let $C$ be a circuit of a $2$-factor of a cubic graph and let $E$ be a subset of edges outside $C$. We denote by $C(E)$ the set of vertices of $C$ that are incident with the edges from $E$. 
If $|C(E)|$ is even, then we can partition the edges of $C$ into two sets $C(E)^A$ and $C(E)^B$  so that every vertex from $C(E)$ is incident with one edge from $C(E)^A$ and one edge from $C(E)^B$ and every vertex of $C$ not from $C(E)$ is incident with two edges from $C(E)^A$ or two edges from $C(E)^B$. We fix $C(E)^A$ and $C(E)^B$ so that $|C(E)^A| \le |C(E)^B|$.

Let $F$ be the rainbow $2$-factor from Lemma~\ref{patternLemma} and let ${\cal R}$, ${\cal G}$, and ${\cal B}$ be the sets of edges coloured by R, G, and B, respectively.
To obtain the first cycle cover we follow the construction from Theorem 17 in \cite{kaiser}.
We define three cycles ${\cal C}_1 = {\cal R} \cup {\cal G} \cup E_1$, ${\cal C}_2 = {\cal R} \cup {\cal B} \cup E_2$, and ${\cal C}_3 = {\cal G} \cup {\cal B} \cup E_3$, where
$E_1$ is $\bigcup_{C \in F} C({\cal R} \cup {\cal G})^A$, 
$E_2$ contains for each circuit either  $C({\cal R} \cup {\cal B})^A$ or $C({\cal R} \cup {\cal B})^B$ depending on the size of the intersection with $E_1$ (we choose the one with the smaller intersection), 
and $E_3$ contains the edges of $F$ that are contained in both $E_1$ and $E_2$ or that are not contained in any of them. One can see that for each circuit $C$ of $F$, the edges $E_3$ contain either $C({\cal G} \cup {\cal B})^A$ or $C({\cal G} \cup {\cal B})^B$.
Let $d_i$ be the number of $i$-circuits in $F$ and let $v_i$ be the number of vertices in $i$-circuits in $F$. 

\begin{lemma}
\label{1CoverLength}
The total length of the first cover is at most
\begin{eqnarray*}
& &2 v_2 + 2  v_4 + 12/5 \cdot v_5 + 14/6 \cdot v_6 + 16/7 \cdot v_7 + 18/8 \cdot v_8 
+ 22/9 \cdot v_9 + 24/10 \cdot v_{10} +\\
& &+ 26/11 \cdot v_{11} + \left(\sum_{i=12}^\infty 5/2 \cdot v_i \right) - 2|{\cal P}_2 \cup {\cal P}_3|.
\end{eqnarray*}
\end{lemma}
\begin{proof}
First, we bound the number of edges from $F$ are used in the cover.
Each edge of $F$ is covered either one or three times. The number of edges covered three times in each circuit $C$ of $F$ cannot exceed $\lfloor|C|/4\rfloor$, because for $C_1$ the number of edges of $C({\cal R} \cup {\cal G})^A$ is at most $|C|/2$ and at most half of these edges are also contained in $C_2$, and hence also in $C_3$. Therefore, the edges from $C$ are used in the cover at most 
$|C|+2\left\lfloor |C|/4 \right\rfloor$ times.

Using Lemma~\ref{patternLemma} we can make improvements of this bound for $4$, $5$, and $8$-circuits. 
All $4$-circuits in $F$ have the pattern RRRR or RRGG and in both cases $4$ edges are enough to cover these circuits. (The length of the cover on circuits does not change with the permutation of colours.)
At least $|{\cal P}_2 \cup {\cal P}_3|$ circuits of length $5$ have pattern RRRGB and $5$ edges are enough to cover these circuits. Similarly, it can be shown for $8$-circuits that the sixteen patterns from the lemma statement guarantee that we can cover each circuit with at most $10$ edges. (Details can be found in \cite{kaiser}.)
The edges of $F$ are covered at most
\begin{eqnarray*}
& & 2d_2 + 4d_4 + 7d_5 + 8d_6 + 9d_7 + 10d_8 + 13d_9 + 14d_{10} + 15d_{11} +\\
& &+  \left(\sum_{i=12}^\infty 3i/2 \cdot d_i \right) - 2|{\cal P}_2 \cup {\cal P}_3|
\end{eqnarray*}
times and since $d_i = v_i/i$, we have
\begin{eqnarray*}
& & v_2 + v_4 + 7/5 \cdot v_5 + 8/6 \cdot v_6 + 9/7 \cdot v_7 + 10/8 \cdot v_8 + 13/9 \cdot v_9 + 14/10 \cdot v_{10} + 15/11 \cdot v_{11} + \\
& &+ \left(\sum_{i=12}^\infty 3/2 \cdot v_i\right) - 2|{\cal P}_2 \cup {\cal P}_3|.
\end{eqnarray*}
All edges in $G/F$ are used exactly twice, thus we increase the above value by $2d_2 + \sum_{i=4}^\infty id_i = v_2 + \sum_{i=4}^\infty v_i$  and the statement of the lemma follows.
\end{proof}

\subsection{The second cover}
Let $F$ be the rainbow $2$-factor from Lemma~\ref{patternLemma} and let ${\cal R}$, ${\cal G}$, and ${\cal B}$ be the sets of edges coloured by R, G, and B, respectively. The colouring associated with $F$ will be called the \emph{starting colouring}. 
If the graph $G/F$ contains R or G cycles, then we recolour such cycles to B until both 
${\cal R}$ and ${\cal G}$ induce acyclic graph in $G/F$. The new colouring will be called the \emph{modified colouring}. 
The $2$-factor with the modified colouring remains a rainbow $2$-factor. 

Let ${\cal R}^0$, ${\cal G}^0$, and ${\cal B}^0$ be the sets of edges coloured by R, G, and B, respectively in the modified colouring.
We construct the cover as follows. The first cycle consists of R and G edges and $C({\cal R}^0 \cup {\cal G}^0)^A$ for every circuit of $F$. The second cycle consists of R and G edges and $C({\cal R}^0 \cup {\cal G}^0)^B$ for every circuit of $F$. 
The third cycle consists of R and B edges and $C({\cal R}^0 \cup {\cal B}^0)^A$ for every circuit of $F$.

\begin{lemma}
\label{2CoverLength}
The total length of the second cover is at most
\begin{eqnarray*}
& & 7/3 \cdot v_2 + 5/2 \cdot v_4 + 5/2 \cdot v_5 + 7/3 \cdot v_6 + 33/14 \cdot v_7 + 19/8 \cdot v_8 + 41/18 \cdot v_9 + 23/10 \cdot v_{10} +\\  
& & + 49/22 \cdot v_{11}+ \sum_{i=12}^\infty 9/4 \cdot v_i.
\end{eqnarray*}
\end{lemma}
\begin{proof}
Let $i_{RC}=1$ if $C$ is not incident to a R edge and $i_{RC}=0$ otherwise.
Let $i_{GC}=1$ if $C$ is not incident to a G edge and $i_{GC}=0$ otherwise.
Let $i_R$ be the number of vertices of $G/F$ that are incident to no R edge and 
let $i_G$ be the number of vertices of $G/F$ that are incident to no B edge.
Let $\text{act}(C)$ be the length of the cover on edges of $C$ and
let $\text{act}(F)$ be the length of the cover on edges of $F$.

Since the R edges induce an acyclic graph in $G/F$, the number of R edges is at most $V(G/F)-i_R-1$.
Since the G edges induce an acyclic graph in $G/F$, the number of G edges is at most $V(G/F)-i_G-1$.
Therefore, we can bound the total number of edges of $G/F$ used in the cover by $E(G/F)+2(V(G/F)-i_R)+V(G/F)-i_G$
and the total number of edges in the cover by $\text{act}(F)+E(G/F)+2(V(G/F)-i_R)+V(G/F)-i_G$,
which can be also expressed by 
\begin{eqnarray}
\sum_{C\in F} \left( \text{act}(C)+\frac{|C|}{2}+3-2i_{RC}-i_{GC} \right). \label{eqnc21}
\end{eqnarray}
For every $C\in F$ we are going to prove that 
\begin{eqnarray}
\text{act}(C)+\frac{|C|}{2}+3-2i_{RC}-i_{GC} \le \left\lfloor \frac{3|C|}{2} \right\rfloor + \frac{|C|}{2}+ 3 - P(C,F). \label{eqnc22}
\end{eqnarray}
All edges of $F$ are covered once by 
the first two cycles of the second cover ant at most $|C|/2$ of the edges of each circuit $C$ in $F$ by the third cycle. Therefore, $\text{act}(C) \le \left\lfloor 3|C|/2 \right\rfloor$.
The improvement function $P(C,F)$ is equal to zero for circuits of lengths other than $2$, $4$, and $6$, and since $i_{RC}$ and $i_{RG}$ are non-negative, the inequality (\ref{eqnc22}) holds for such circuits. Note that to calculate $P(C,F)$ we use the starting colouring not the modified one. 

Next, we consider circuits of length $2$, $4$, and $6$.
Let $C$ be a circuit of length $2$. If $C$ has type BB in the modified colouring, then we need $3$ edges to cover $C$.
Note that $i_{RC}=1$ and $i_{GC}=1$. Recall that $P(C) \le 3$ (the pattern in the starting colouring may be RR) and 
the inequality (\ref{eqnc22}) holds.  
If $C$ has type GG in the modified colouring, then we need $2$ edges to cover $C$.
Note that $i_{RC}=1$ and $i_{GC}=0$. Recall that $P(C) = 3$ and the inequality (\ref{eqnc22}) holds.  
If $C$ has type RR in the modified colouring, then we need $3$ edges to cover $C$.
Note that $i_{RC}=0$ and $i_{GC}=1$. Recall that $P(C) = 1$ and
the inequality (\ref{eqnc22}) holds.  

Let $C$ be a circuit of length $4$. Then $C$ cannot be of type RGRG neither in the starting colouring nor in the modified colouring, because 
to obtain the modified colouring we only recoloured some of the edges to the colour B. 
The following table lists all possible types of $C$ in the modified colouring. 
In all cases the inequality (\ref{eqnc22}) holds.  

\begin{center}
\begin{tabular}{ccccc}
type & $\text{act}(C)$ & $i_{RC}$ & $i_{GC}$ & $P(C)$\\\hline
BBBB & 6 & 1 & 1 & 1 \\
GGGG & 4 & 1 & 0 & 1 \\
RRRR & 6 & 0 & 1 & 1 \\
RRGG & 5 & 0 & 0 & 1 \\
RRBB & 6 & 0 & 1 & 1 \\
RBRB & 6 & 0 & 1 & 1 \\
GGBB & 5 & 1 & 0 & 1 \\
GBGB & 6 & 1 & 0 & 1 
\end{tabular}
\end{center}

Let $C$ be a circuit of length $6$. If $i_{RC}=1$, then even when $P(C) = 2$, the inequality (\ref{eqnc22}) holds.
If $i_{RC}=0$ and $i_{GC}=1$, then $P(C) \le 1$ and the inequality (\ref{eqnc22}) holds.
Therefore, suppose that $i_{RC}=0$ and $i_{GC}=0$. 
If the colour B is missing in the starting colouring of $C$, then $P(C)=0$ because both R and G colour must be present in the starting colouring in order to get both colours in the modified colouring. So the inequality (\ref{eqnc22}) holds. 
This leaves us with the case when all three colours are present in both the starting and the modified colouring. According to Lemma \ref{patternLemma} only patterns RRGGBB and RRGBBG  remain to be considered. There are only four types associated with these two patterns and in each case the inequality (\ref{eqnc22}) holds:
\begin{center}
\begin{tabular}{ccccc}
type & $\text{act}(C)$ & $i_{RC}$ & $i_{GC}$ & $P(C)$\\\hline
RRGGBB & 8 & 0 & 0 & 1\\
RRGBBG & 8 & 0 & 0 & 1\\
RRBGGB & 8 & 0 & 0 & 1\\
RGGRBB & 8 & 0 & 0 & 1\\
\end{tabular}
\end{center}
This concludes the proof of (\ref{eqnc22}). Together with (\ref{eqnc21}) the total length of the cover at $F$ is at most
$$
\sum_{C\in F} \left(\left\lfloor \frac{3|C|}{2} \right\rfloor + \frac{|C|}{2}+ 3 - P(C,F) \right)=
 - P(F) + \sum_{i=2}^\infty \left( \left\lfloor \frac{3i}{2} \right\rfloor + \frac{i}{2} + 3\right) \cdot d_i.
$$
From Lemma \ref{patternLemma} we have $P(F) \ge 7/3 \cdot d_2 + d_4 + d_6$. 
Therefore, the length of the cover is at most
$$
 14/3 \cdot d_2 + 10 d_4+ 25/2 \cdot d_5 + 14 d_6 + 
\sum_{i=7}^\infty \left(\lfloor 3i/2 \rfloor + 3\right) \cdot d_i.
$$
Taking $d_i=i \cdot v_i$ we have
\begin{eqnarray*}
& &7/3 \cdot v_2 + 5/2 \cdot v_4 + 5/2 \cdot v_5 + 7/3 \cdot v_6 + 33/14 \cdot v_7 + 19/8 \cdot v_8 + 41/18 \cdot v_9 + 23/10 \cdot v_{10} + \\
&+& 49/22 \cdot v_{11} + \sum_{i=12}^\infty 9/4 \cdot v_i
\end{eqnarray*}
which concludes the proof of the lemma.
\end{proof}

\subsection{Proof of Theorem~\ref{thmmain}}
\begin{proof}[Proof of Theorem~\ref{thmmain}]
Lemma \ref{1CoverLength} and \ref{2CoverLength} bound the total lengths of the two cycle covers that we constructed. Since the length of the shortest cycle cover does not exceed either of the two bounds, we can bound it by a convex combination of the two bounds. We use a ratio of $1/3:2/3$. The combination is as follows.

\begin{align*}
1/3  \cdot \biggl[&
2v_2 + 2v_4 + 12/5 \cdot v_5 + 14/6 \cdot v_6 + 16/7 \cdot v_7 + 18/8 \cdot v_8 + 22/9 \cdot v_9 \\
&\qquad + 24/10 \cdot v_{10} + 26/11 \cdot v_{11} + \left( \sum_{i=12}^\infty 5/2 \cdot v_i \right) - 2|{\cal P}_2 \cup {\cal P}_3| \biggr] \\
+ 2/3 \cdot \biggl[& 7/3 \cdot v_2 + 5/2 \cdot v_4 + 5/2 \cdot v_5 + 7/3 \cdot v_6 + 33/14 \cdot v_7 + 19/8 \cdot v_8 + 41/18 \cdot v_9\\ 
&\qquad + 23/10 \cdot v_{10} + 49/22 \cdot v_{11} + \sum_{i=12}^\infty 9/4 \cdot v_i \biggr] \\
= \qquad & 20/9 \cdot v_2 + 7/3 \cdot v_4 + 37/15 \cdot v_5 + 7/3 \cdot v_6 + 7/3 \cdot v_7 + 7/3 \cdot v_8 + 7/3 \cdot v_9 \\
&\qquad+ 7/3 \cdot v_{10} + 25/11 \cdot v_{11} + \left(\sum_{i=12}^\infty 7/3 \cdot v_i\right) -  2/3 |{\cal P}_2 \cup {\cal P}_3|\\
\le \qquad & 37/15 \cdot v_5 + 7/3 \cdot (|V(G)|-v_5) - 2/3 \cdot |{\cal P}_2 \cup {\cal P}_3|.
\end{align*}
Therefore, the length of the shortest cover is at most
\begin{eqnarray}
7/3 \cdot |V(G)| + 2 /15 \cdot v_5 - 2/3 \cdot |{\cal P}_2 \cup {\cal P}_3|. \label{mmres}
\end{eqnarray}
Since by Lemma~\ref{patternLemma} (when we choose the $2$-factor $F_1$) we have 
$v_5 \le 1/2 \cdot |V(G)| + 1/3 \cdot |{\cal P}_2| + 1/10 \cdot |{\cal P}_3|$ 
we get that the total length of the shortest cycle cover is at most
\begin{align*}
& 7/3 \cdot |V(G)| + 2 /15 \cdot (1/2 \cdot |V(G)| + 1/3 \cdot |{\cal P}_2| + 1/10 \cdot |{\cal P}_3|) 
- 2/3 \cdot |{\cal P}_2 \cup {\cal P}_3| \\
&\le 12/5 \cdot |V(G)| = 8/5 \cdot |E(G)| = 1.6\cdot |E(G)|. \qedhere
\end{align*}
\end{proof}

\begin{theorem}
\label{thm5Circuit}
Every bridgeless cubic graph with $m$ edges and at most $k$ circuits of length $5$ has a cycle cover of length at most 
$14/9 \cdot m + 1/9 \cdot k$.
\end{theorem}
\begin{proof}
By Lemma~\ref{patternLemma}, choosing the $2$-factor $F_2$, we have that $v_5 =5 d_5 \le 5 \cdot 1/6 \cdot k$. From (\ref{mmres}) we get that the cover has size at most
$7/3 \cdot |V(G)| + 1/9 \cdot k = 14/9 \cdot m + 1/9 \cdot k$.
\end{proof}

\begin{corollary}
\label{corDisjoint5}
Every bridgeless cubic graph with $m$ edges, such that all circuits of length $5$ are disjoint, 
has a cycle cover of length at most $212/135 \cdot m \approx 1.570 m$.
\end{corollary}

\begin{corollary}
\label{corWithout5}
Every bridgeless cubic graph with $m$ edges without circuits of length $5$ has a cycle cover of length at most $14/9 \cdot m \approx 1.556 m$.
\end{corollary}

\end{document}